\documentclass[10pt]{article}

\usepackage{mathrsfs}
\usepackage{amssymb,amsthm,amsmath,hyperref}

\usepackage{graphicx,float}

\numberwithin{equation}{section}

\newtheorem{thm}{Theorem}[section]
\newtheorem{lem}{Lemma}[section]
\newtheorem{cor}{Corollary}[section]

\theoremstyle{definition}
\newtheorem{defn}{Definition}[section]
\theoremstyle{remark}
\newtheorem{rem}{Remark}[section]

\allowdisplaybreaks

\begin{document}

\title{Decay Estimates for Isentropic Compressible Navier-Stokes Equations
in Bounded Domain }

\author{Daoyuan Fang\thanks{Email: dyf@zju.edu.cn}, Ruizhao
Zi\thanks{Email: ruizhao3805@163.com}, Ting Zhang\thanks{Email:
zhangting79@zju.edu.cn}\\\textit{\small Department of Mathematics,
Zhejiang University, Hangzhou 310027, China}}

\date{}
\maketitle

\begin{abstract}
In this paper, under the hypothesis that $\rho$ is upper bounded, we
construct a Lyapunov functional for the multidimensional isentropic
compressible Navier-Stokes equations and show that the weak
solutions decay exponentially to the equilibrium state in $L^2$
norm. This can be regarded as a generalization of Matsumura and
Nishida's results in \cite{Matsumura 3}, since our analysis is done
in the framework of Lions \cite{Lions 2} and Feireisl et al.
\cite{Feireisl 3}, the higher regularity of $(\rho, u)$ and the
uniformly positive lower bound of $\rho$ are not necessary in our
analysis and vacuum may be admitted. Indeed, the upper bound of the
density $\rho$ plays the essential role in our proof.
\end{abstract}

\section{Introduction}
This paper is devoted to the asymptotic behavior of the solutions to
the Navier-Stokes equations of an isentropic compressible fluid:
\begin{eqnarray}\label{1.1}
\left\{ \begin{array}{ll}
\rho_t+\textrm{div}(\rho u)=0,\\[3mm]
(\rho u)_t+\textrm{div}(\rho u\otimes u)+\nabla P(\rho)=\mu\Delta
u+(\lambda+\mu)\nabla \textrm{div} u,
 \end{array}
 \right.
\end{eqnarray}
where the density $\rho=\rho(t,x)$ and the velocity
$u=(u^1(t,x),u^2(t,x),\cdots,u^N(t,x))$ are functions of the time
$t\in(0,\infty)$ and the spatial coordinate $x\in \Omega$ where
$\Omega\subset\mathbb{R}^N$, $N\geq2$, is a bounded regular domain.
$P(\rho)=a\rho^\gamma$ is the pressure, with $a>0$ and $\gamma>1$
being two positive constants. Since the constant $a$ does not play
any role in the analysis, we assume henceforth that $a=1$. The
constants $\mu$ and $\lambda$ are viscosity coefficients satisfying
$$
\mu>0,\ \lambda+\frac{2}{N}\mu\geq0.
$$
We prescribe the initial conditions for the density and momenta:
\begin{eqnarray}\label{1.2}
\rho(0)=\rho_0,\ (\rho u)(0)=m_0,
\end{eqnarray}
together with the no-slip boundary conditions for the velocity:
\begin{eqnarray}\label{1.3}
u|\partial\Omega=0.
\end{eqnarray}

Some of the previous works in this direction can be summarized as
follows. The first general result on weak solutions to the
multidimensional isentropic compressible Navier-Stokes equations
with large initial data was obtained by Lions in \cite{Lions 2}, in
which he used the renormalization skills introduced by DiPerna and
Lions in \cite{DiPerna} to obtain global weak solutions provided
that the specific heat ratio $\gamma$ is appropriately large, for
example, $\gamma\geq3N/(N+2),N=2,3$. Later, Feireisl,  Novotn\'{y}
and Petzeltov\'{y} \cite{Feireisl 3} improved  Lions's result to the
case $\gamma>\frac{N}{2}$. If the initial data was assumed to have
some symmetric properties, Jiang and  Zhang \cite{Jiang 1, Jiang 2}
obtained the global weak solutions for any $\gamma>1$. For the full
Navier-Stokes equations,  Feireisl and  Petzeltov\'{a} established
the variational solutions, see \cite{Feireisl 4}, for example.

Concerning the large time behavior of solutions to the
initial-boundary value problem (\ref{1.1})-(\ref{1.3}), by using the
weak convergence method,  Feireisl and  Petzeltov\'{a}
\cite{Feireisl 1} proved the weak solutions to the problem
(\ref{1.1})-(\ref{1.3}) with a gradient external force $\nabla F$
independent of time $t$ converge to stationary solution $(\rho_s,
0)$ in the following sence
$$
\rho(t)\rightarrow\rho_s\ \textrm{strongly\  in}\ L^\gamma(\Omega),\
\textrm{ess}\
\sup_{\tau>t}\int_\Omega\rho(\tau)|u(\tau)|^2dx\rightarrow0,\
\textrm{as}\ t\rightarrow\infty,
$$
where the domain $\Omega$ need not to be bounded and the initial
data need not to be close to the equilibrium state. If the initial
data is close to the equilibrium state, there are many results on
the problem of large time behavior of global smooth solutions to the
compressible Navier-Stokes equations (of heat-conducting flow). When
there is no external or internal force involved, the $H^s$ global
existence and time-decay rate of strong solutions are obtained in
whole space $\mathbb{R}^3$ first by Matsumura and Nishida \cite{Matsumura 1,Matsumura 2}
and the optimal $L^p \ (p\geq2)$ decay rate is established by Ponce
\cite{Ponce}. The large time decay rate of global solution in
multi-dimensional half space or exterior domain is also investigated
for the compressible Navier-Stokes equations by Kagei and Kobayashi
\cite{Kagei 1, Kagei 2}, Kobayashi and Shibata \cite{Kobayashi 1},
and Kobayashi \cite{Kobayashi 2}. Therein, the optimal $L^2$
time-decay rate in three dimension is established as
$$
\|(\rho-\tilde{\rho},u)(t)\|_{L^2(\mathbb{R}^3)}\leq C(1 +
t)^{-\frac{3}{4}},
$$
with $( \tilde{\rho}, 0)$ the constant state, under small initial
perturbation in Sobolev space. When additional (exterior or
internal) potential force is taken into account, the global
existence of a strong solution and convergence to steady state are
investigated by Matsumura and Nishida \cite{Matsumura 4} and many
other authors \cite{Deckelnick 1, Deckelnick 2,  Shibata 1, Shibata
2, Ukai}.
 The optimal $L^p$ convergence rate in
$\mathbb{R}^3$ is established by Duan et al. \cite{Duan} for the
non-isentropic compressible flow as
$$
\|(\rho-\tilde{\rho}, u, \theta -\theta_\infty)(t)
\|_{L^p(\mathbb{R}^3)} \leq C(1 + t)^{- \frac32 (1-\frac1p)}, \
2\leq p\leq 6,
$$
where $( \tilde{\rho}, 0, \theta_\infty)$ is related to the
steady-state solution, under the same smallness assumptions on
initial perturbation and the external force. If $\Omega$  is a
smooth bounded domain in $\mathbb{R}^3$, based on high order energy
estimates, Matsumura and Nishida \cite{Matsumura 3} proved that, for
large times, the solution decays exponentially to a unique
equilibrium state. However, for the one dimensional case, the
smallness assumption on the initial data and force can be removed.
Indeed, by constructing suitable Lyapunov functionals, the decay
rate estimates in $L^2$ norm and $H^1$ norm are established by
Stra\u{s}kraba and Zlotnik \cite{Straskraba 1}, and the decay is
exponential if so the decay rate of the nonstationary part of the
mass force is.

In this paper, under the hypothesis that $\rho$ is upper bounded, we
construct a Lyapunov functional for the multidimensional isentropic
compressible Navier-Stokes equations (\ref{1.1}) with the aid of the
operator $\mathcal{B}$ introduced by Bogovskii \cite{Bogovskii} (cf.
Lemma \ref{2.5}). Based on this, we show that the weak solutions to
problem (\ref{1.1}) decay exponentially to the equilibrium state in
$L^2$ norm. The ideas mainly come from \cite{Straskraba 1}, however,
unlike \cite{Straskraba 1}, we do not divide by $\rho$ on both side
of $(\ref{1.1})_2$. As a result, the uniformly positive lower bound
of $\rho$ is not necessary in our analysis and vacuum may be
admitted. Compared with \cite{Matsumura 3}, our analysis follows the
framework of Lions \cite{Lions 2} and Feireisl et al. \cite{Feireisl
3}, and thus the higher regularity of $(\rho, u)$ is not necessary
here. Actually, the upper bound of the density $\rho$ plays the
essential role in our proof. Coincidentally, a blow-up criterion for
the 3D compressible Navier-Stokes equations was given in terms of
the upper bound of the density $\rho$ by Sun et al. \cite{Sun},
however, their result does not contain the case for spatial
dimension $N>3$.

 Now we give a precise formulation of our result. Let
$\rho_s$ be the solution of the following stationary problem:
\begin{eqnarray}\label{1.4}
\left\{
\begin{array}{l}
\nabla P(\rho_s)=0,\\[3mm]
\displaystyle\int_\Omega \rho_sdx=\int_\Omega\rho_0dx.
\end{array}
\right.
\end{eqnarray}
Then $\displaystyle\rho_s=\frac{1}{|\Omega|}\int_\Omega\rho_0dx$ be
a positive constant.

Formally, the total energy of problem (1.1) can be written as
$$
E(t)=\int_\Omega\frac12\rho(t)|u(t)|^2+\frac{1}{\gamma-1}\rho^\gamma(t)dx,
$$
satisfying the energy inequality
\begin{eqnarray}\label{1.5}
\frac{dE}{dt}+\int_\Omega\mu|\nabla
u|^2+(\lambda+\mu)(\textrm{div}u)^2dx\leq0.
\end{eqnarray}

The definition of weak solutions to the problem
(\ref{1.1})-(\ref{1.3}) is given as follows:

\begin{defn}[\textbf{finite energy weak solutions},\cite{Lions 2, Feireisl 3}]
A pair of function $(\rho,u)$ will be termed a finite energy weak
solution of the problem (1.1), (1.3) on $(0,\infty)\times\Omega$, if

\begin{itemize}
  \item $\rho\geq0, \rho\in L^\infty_{loc}(0,\infty;L^\gamma(\Omega)),
u\in L^2_{loc}(0,\infty;W_0^{1,2}(\Omega))$.
  \item The equations (1.1) are satisfied in
$\mathscr{D}'((0,\infty)\times\Omega)$; moreover, $(1.1)_1$ holds in
$\mathscr{D}'((0,\infty)\times\mathbb{R}^N)$ provided $\rho, \ u$
were prolonged to be zero on $\mathbb{R}^N\backslash\Omega$.
  \item The energy inequality (1.4) holds in $\mathscr{D}'(0,\infty)$.
  \item The equality $(1.1)_1$ holds in the sense of  renormalized
solutions, more precisely, the following equation
$$
b(\rho)_t+\textrm{div}(b(\rho)u)+(b'(\rho)\rho-b(\rho))\textrm{div}u=0
$$
holds in $\mathscr{D}'((0,\infty)\times\Omega)$ for any $b\in
C^1(\mathbb{R})$ such that
$$
b'(z)=0 \ for \ all\ z\in\mathbb{R}\  \textrm{large\  enough,\ say},
|z|\geq M,
$$
where the constant $M$ may vary for different functions $b$.
\end{itemize}
\end{defn}

The paper is mainly concerned with the proof of the following
theorem.

\begin{thm}\label{thm1}
Let $\Omega$ be a bounded  Lipschitz domain in $\mathbb{R}^N$,
$N\geq2$, $(\rho, u)$ be a finite energy weak solution to the
problem (\ref{1.1})-(\ref{1.3}) and $\rho_s$ be the solution to the
stationary problem (\ref{1.4}). In addition, if we assume $\rho$ is
upper bounded, i.e., there exists constant $\bar{\rho}>0$, such that
$$
\rho\leq\bar{\rho},\ a.e.\ (t,x)\in(0,\infty)\times\Omega
$$
Then
\begin{equation}\label{1.6}
\int_\Omega\rho|u|^2+(\rho-\rho_s)^2dx\leq C(E_0,
\bar{\rho})\exp\{-C(\bar{\rho},\Omega)t\} \ for\ a.e.\
t\in(0,\infty),
\end{equation}
where
$$
E_0=\int_\Omega\frac{|m_0|^2}{2\rho_0}+\frac{1}{\gamma-1}\rho^\gamma_0dx.
$$
\end{thm}

\begin{rem}\label{rem2}
A natural question is that whether the solution $(\rho,  u)$ stated
in Theorem \ref{thm1} exists. Indeed, Matsumura and Nishida
\cite{Matsumura 3} obtained the existence of global solution to the
compressible heat-conductive fluid in bounded domain in
$\mathbb{R}^3$ provided the initial data is close to the equilibrium
state. If vacuum is taken into account, Huang, Li and  Xin
\cite{Huang 2} established the global existence and uniqueness of
classical solutions to the Cauchy problem for the isentropic
compressible Navier-Stokes equations in three spatial dimensions
with smooth initial data which are of small energy but possibly
large oscillations with constant state as far field which could be
either vacuum or non-vacuum.

\end{rem}

\begin{rem}\label{rem4}
We believe that our method can be adapted to the other related
models. This is the object of our future work.
\end{rem}
\noindent\textbf{Notations:}
\begin{enumerate}
  \item $\eta_\epsilon(\cdot)=\frac{1}{\epsilon^N}\eta(\frac{\cdot}{\epsilon})$,
where $\eta$ is the standard mollifier in $\mathbb{R}^N$.
  \item $[f]_\epsilon=\eta_\epsilon\ast f$, for any $f\in
L^1_{loc}(\mathbb{R}^N)$.
\end{enumerate}

 The rest of the paper is organized as follows: In section 2, we present some preliminary
 results which will be used later. In section 3, we give the proof
 of Theorem \ref{thm1}.

 .
\section{Preliminaries}
\begin{lem}[\cite{Feireisl 3}]\label{lem1}
 Let $\rho,\ u$ be a solution of $(1.1)_1$ in
$\mathscr{D}'((0,\infty)\times\Omega)$ and such that $\rho\in
L^2((0,\infty)\times\Omega)$ and $u\in
L^2(0,\infty;[W_0^{1,2}(\Omega)]^N) $.

Then, prolonging $\rho,\ u$ to be zero on
$\mathbb{R}^N\backslash\Omega$, the equation $(1.1)_1$ holds in
$\mathscr{D}'((0,\infty)\times\mathbb{R}^N).$
\end{lem}

\begin{lem}[\cite{Feireisl 3}]\label{lem2}
 Let $(\rho,u)$ be a finite energy weak solution
of problem (1.1)-(1.3) on the time interval $(0,\infty)$.

Then the total mass
$m[\rho(t)]\doteq\displaystyle\int_\Omega\rho(t)dx$ is conserved,
i.e.,
\begin{eqnarray}\label{2.1}
\int_\Omega\rho(t)dx=\int_\Omega\rho_0dx,
\end{eqnarray}
for all $t\in(0,\infty)$.
\end{lem}

\begin{lem}[\cite{Lions 1, Feireisl 4}] \label{lem3}
Let $\Omega\subset\mathbb{R}^N$ be a
domain and $\rho\in L^p(\Omega), u\in [W^{1,q}(\Omega)]^N$ be given
functions with $1\leq p,q\leq\infty$ and
$\frac{1}{p}+\frac{1}{q}\leq1$.

Then for any compact $K\subset\Omega$,

(i)
\begin{eqnarray}\label{2.2}
\|[\textrm{div}(\rho u)]_\epsilon-\textrm{div}([\rho]_\epsilon
u)\|_{L^r(K)}\leq c(K)\|\rho\|_{L^p(\Omega)}\|u\|_{W^{1,q}(\Omega)}
\end{eqnarray}
provided $\epsilon$ is small enough, where $\frac
1r=\frac1p+\frac1q$. In addition, if $\Omega=\mathbb{R}^N$, $K$ can
be replaced by $\mathbb{R}^N$.

(ii)
\begin{eqnarray}\label{2.3}
[\textrm{div}(\rho u)]_\epsilon-\textrm{div}([\rho]_\epsilon
u)\rightarrow0 \ \textrm{in}\  L^\theta(K)\ \ as\
\epsilon\rightarrow0,
\end{eqnarray}
where $\frac 1\theta=\frac1p+\frac1q$, if $p<\infty$ and
$1\leq\theta<q$ if $p=\infty$. In addition, if $\Omega=\mathbb{R}^N$
and $p<\infty$, $K$ can be replaced by $\mathbb{R}^N$.
\end{lem}

\begin{proof}
Since the proof of most of the results in this lemma can be found in
\cite{Lions 1, Feireisl 4}, here we only prove (ii) for the case
$p=\infty$. To this end, we define
$G_\epsilon(\rho)=[\textrm{div}(\rho
u)]_\epsilon-\textrm{div}([\rho]_\epsilon u)$. Choosing any open
subset $U$ such that $K\subset U\subset\subset\Omega$, then $\rho\in
L^{\infty}(\Omega)$ implies $\rho\in L^{\tilde{p}}(U)$, where
$\tilde{p}=\frac{q\theta}{q-\theta}$ satisfying
$\frac{1}{\tilde{p}}+\frac{1}{q}=\frac{1}{\theta}$. It is easy to
see that $G_\epsilon(\rho)\rightarrow0$ in $L^\theta(\Omega)$ as
$\epsilon\rightarrow0$ for any $\rho\in C^\infty_0(\Omega)$. Now
choosing a sequence $\rho_n\in C_0^\infty(U)$ such that
$\rho_n\rightarrow\rho$ in $L^{\tilde{p}}(U)$ as
$n\rightarrow\infty$, using the result in (i) with $\rho,\ p, \ r$
and $\Omega$ replaced by $\rho-\rho_n,\ \tilde{p},\ \theta$ and $U$,
respectively, we have
$$
\begin{array}[b]{rl}
\|G_\epsilon(\rho)\|_{L^{\theta}(K)}&\leq\|G_\epsilon(\rho-\rho_n)\|_{L^{\theta}(K)}
+\|G_\epsilon(\rho_n)\|_{L^{\theta}(K)}\\[3mm]
&\leq
c(K)\|\rho-\rho_n\|_{L^{\tilde{p}}(U)}\|u\|_{W^{1,q}(U)}+\|G_\epsilon(\rho_n)\|_{L^{\theta}(K)}\\[3mm]
&\leq
c(K)\|\rho-\rho_n\|_{L^{\tilde{p}}(U)}\|u\|_{W^{1,q}(\Omega)}+\|G_\epsilon(\rho_n)\|_{L^{\theta}(\Omega)}\\[3mm]
&\rightarrow 0,\ \textrm{as}\ \epsilon\rightarrow0,\
n\rightarrow\infty.
\end{array}
$$
This completes the proof of Lemma \ref{lem3}.
\end{proof}

\begin{cor}\label{cor1}
 If $\rho\in L^{\infty}((0,\infty)\times\Omega)$ and
$u\in L^2(0,\infty;[W_0^{1,2}(\Omega)]^N) $ solves $(1.1)_1$ in
$\mathscr{D}'((0,\infty)\times\Omega)$.

Then for any subset $[\alpha,\beta]\subset(0,\infty)$, we have
\begin{eqnarray}\label{2.4}
\partial_t[\rho]_\epsilon+\textrm{div}([\rho]_\epsilon u)=r_\epsilon\ \ a.e.\
on\  [\alpha,\beta]\times\Omega,
\end{eqnarray}
where $r_\epsilon=\textrm{div}([\rho]_\epsilon u)-[\textrm{div}(\rho
u)]_\epsilon$. Moreover, $r_\epsilon$ is bounded in
$L^2([\alpha,\beta]\times\Omega)$ uniformly in $\epsilon$ and
converges to 0 strongly in $L^{2}(\alpha,\beta; L^{\theta}(\Omega))$
for all $\theta\in[1,2)$.
\end{cor}
\begin{proof}
Firstly, by virtue of Lemma \ref{lem1}, the equation $(1.1)_1$ holds
in $\mathscr{D}'((0,\infty)\times\mathbb{R}^N)$ provided $\rho,\ u$
were extended to be zero on $\mathbb{R}^N\backslash\Omega$. Then we
use the mollifier $\eta_\epsilon$ as test functions to deduce
(\ref{2.4}) provided $\epsilon>0$ is small enough. It follows from
Lemma \ref{lem3} immediately that $r_\epsilon$ is bounded in
$L^2([\alpha,\beta]\times\Omega)$ uniformly in $\epsilon$, together
with Lebesgue's dominated convergence theorem, we have
$r_\epsilon\rightarrow0 \ \textrm{in}\ L^{2}(\alpha,\beta;
L^{\theta}(\Omega))$ for $\theta\in[1,2)$, where we have used the
fact that $\Omega$ is bounded.
\end{proof}

\begin{lem}[\cite{Feireisl 2, Feireisl 3}]\label{lem5}
 Let $\Omega$ be a bounded Lipschitz
domain in $\mathbb{R}^N$, $N\geq2$, and $p,\  r\in(1, \infty)$ given
numbers,  $f\in \{L^p(\Omega)|\int_\Omega fdx=0$\}.

Then the problem
\begin{equation}\label{2.5}
\textrm{div}v=f, v|\partial\Omega=0,
\end{equation}
admits a solution operator $\mathcal {B}: f\mapsto v$ enjoying the
following properties:
\begin{itemize}
  \item $\mathcal {B}$ is a linear operator from $L^p(\Omega)$ into
  $[W^{1,p}_0(\Omega)]^N$, i.e.,
  $$
\|\mathcal{B}[f]_{W^{1,p}(\Omega)}\|\leq c(p,
\Omega)\|f\|_{L^p(\Omega)};
  $$
  \item The function $v=\mathcal{B}[f]$ solves the problem (\ref{2.5});
  \item If a function $f\in L^p(\Omega)$ can be written in the form
  $f=divg$ with $g\in [L^r(\Omega)]^N$ and $g\cdot n=0$ on
  $\partial\Omega$, where $n$ is the outward pointing unit normal vector field along $\partial\Omega$,  then
  $$
  \|\mathcal{B}[f]\|_{L^r(\Omega)}\leq c(p, r, \Omega)\|g\|_{L^r(\Omega)}.
  $$
\end{itemize}
\end{lem}
\begin{rem}\label{rem2.6}
To our best knowledge, the operator $\mathcal{B}[\cdot]$ was first
constructed by Bogovskii\cite{Bogovskii}. A complete proof of the
above mentioned properties may be found in Galdi \cite{Galdi} or
Borchers and Sohr\cite{Borchers}. Moreover, the operator
$\mathcal{B}[\cdot]$ was first used by Feireisl and Petzeltov\'{a}
\cite{Feireisl 2} to show the existence of weak solutions $(\rho,u)$
to the problem (\ref{1.1})-(\ref{1.3}) with the density $\rho$
square integrable up to the boundary $\partial\Omega$.
\end{rem}
\section{proof of theorem \ref{thm1}}
\begin{lem}\label{lem6}
Let $r_0>0, \bar{r}>0$ and $\gamma>1$ be arbitrary fixed constants,
$\displaystyle f(r)=r\int^r_{r_0}\frac{h^\gamma-r_0^\gamma}{h^2}dh$
for $r\in[0,\bar{r}]$. Then there exists positive constants $K_1$
and $K_2$ depending on $r_0$ and $\bar{r}$, such that
\begin{equation}\label{3.1}
K_1(r-r_0)^2\leq f(r)\leq K_2(r-r_0)^2\ \ for\ all\ r\in[0,\bar{r}].
\end{equation}
\end{lem}
\begin{proof}
Let
$$
g(r)=\frac{\displaystyle
r\int^r_{r_0}\frac{h^\gamma-r_0^\gamma}{h^2}dh}{(r-r_0)^2}.
$$
It is easy to see that
$$
\lim_{r\to 0}g(r)=\frac{\displaystyle\lim_{r\to
0}r\int^r_{r_0}\frac{h^\gamma-r_0^\gamma}{h^2}dh}{r_0^2}=r_0^{\gamma-2}>0,
$$
Using the l'Hospital rule, we obtain
$$
\lim_{r\to r_0}g(r)=\lim_{r\to
r_0}\frac{\displaystyle\int^r_{r_0}\frac{h^\gamma-r_0^\gamma}{h^2}dh+\frac{r^\gamma-r_0^\gamma}{r}}{2(r-r_0)}
=\frac{\gamma}{2}r_0^{\gamma-2}>0.
$$
Consequently, $g(r)$ is a continuous function on $[0,\bar{r}]$ with
$g(r)>0$, and (\ref{3.1}) follows immediately.
\end{proof}
Now we are going to give the proof of Theorem \ref{thm1}. First of
all, we need to rewrite the energy inequality (\ref{1.5}) as a new
form. To this end, we choose an arbitrary
$\psi(t)\in\mathscr{D}(0,\infty)$ with $\psi(t)\geq0$, then the
energy inequality (\ref{1.5}) is equivalent to
$$
-\int^\infty_0\psi_t\int_\Omega\frac12\rho|u|^2+\frac{\rho^\gamma}{\gamma-1}dxdt
+\int^\infty_0\psi\int_\Omega\mu|\nabla
u|^2+(\lambda+\mu)(\textrm{div}u)^2dxdt\leq0.
$$
Recalling that
$\displaystyle\rho_s=\frac{1}{|\Omega|}\int_\Omega\rho_0dx$ is a
positive constant and $\displaystyle\int_\Omega\rho(t)dx$ is
independent of $t$ due to Lemma \ref{lem2}, we thus have
$$
\int^\infty_0\psi_t\int_\Omega\big(-\frac{\gamma}{\gamma-1}\rho\rho_s^{\gamma}+\rho_s^\gamma\big)dxdt=0.
$$
Adding the above two equations, we have
\begin{equation}\label{3.2}
-\int_0^\infty\psi_t\int_\Omega\big(\frac12\rho|u|^2+\rho\int^\rho_{\rho_s}\frac{h^\gamma-\rho^\gamma_s}{h^2}dh\big)dxdt
+\int^\infty_0\psi\int_\Omega\big(\mu|\nabla
u|^2+(\lambda+\mu)(\textrm{div}u)^2\big)dxdt\leq0.
\end{equation}
Next, we use the operator $\mathcal{B}$ introduced in Lemma
\ref{lem5} to construct test function of the form
$$
\Phi(t,x)=\psi(t)\mathcal{B}[[\rho]_\epsilon-\rho^\epsilon_s],
$$
where $\psi$ is the same as in (\ref{3.2}) and
$\displaystyle\rho^\epsilon_s=\frac{1}{|\Omega|}\int_\Omega[\rho]_\epsilon
dx$. Obviously, $\Phi(t,x)$ is smooth in $x$ and vanishes near
$\partial\Omega$ due to the properties of operator $\mathcal{B}$.
Moreover, since $\rho\in L^\infty(0,\infty;\Omega)$, $\Phi_t$ is in
$L^2(0,\infty;[W^{1,2}_0(\Omega)]^N)$ in view of Corollary
\ref{cor1}. Consequently,  $\Phi$ could be used as a test function
for the equation $(\ref{1.1})_2$. Thus, we have
\begin{equation}\label{3.3}
\begin{array}[b]{rl}
&\displaystyle-\int_0^\infty\psi_t\int_\Omega\rho
u\mathcal{B}[[\rho]_\epsilon-\rho^\epsilon_s]dxdt+\int^\infty_0\psi\int_\Omega\rho
u\mathcal{B}[\textrm{div}([\rho]_\epsilon u)]dxdt\\[3mm]
&\displaystyle-\int^\infty_0\psi\int_\Omega\rho
u\mathcal{B}[r_\epsilon-\frac{1}{|\Omega|}\int_\Omega r_\epsilon
dx]dxdt-\int^\infty_0\psi\int_\Omega\rho u\otimes
u:\nabla\mathcal{B}[[\rho]_\epsilon-\rho^\epsilon_s]dxdt\\[3mm]
&\displaystyle-\int^\infty_0\psi\int_\Omega(P(\rho)-P(\rho_s))([\rho]_\epsilon-\rho_s^\epsilon)dxdt+
\mu\int^\infty_0\psi\int_\Omega\nabla
u:\nabla\mathcal{B}[[\rho]_\epsilon-\rho_s^\epsilon]dxdt\\[3mm]
&\displaystyle+(\lambda+\mu)\int^\infty_0\psi\int_\Omega
\textrm{div}u([\rho]_\epsilon-\rho_s^\epsilon)dxdt=0,
\end{array}
\end{equation}
where we have used Corollary $\ref{cor1}$. Letting $\epsilon\to 0$,
we obtain by virtue of Corollary $\ref{cor1}$ and the properties of
operator $\mathcal{B}$,
\begin{equation}\label{3.4}
\begin{array}[b]{rl}
&\displaystyle-\int_0^\infty\psi_t\int_\Omega\rho
u\mathcal{B}[\rho-\rho_s]dxdt+\int^\infty_0\psi\int_\Omega\rho
u\mathcal{B}[\textrm{div}(\rho u)]dxdt\\[3mm]
&\displaystyle-\int^\infty_0\psi\int_\Omega\rho u\otimes
u:\nabla\mathcal{B}[\rho-\rho_s]dxdt\\[3mm]
&\displaystyle-\int^\infty_0\psi\int_\Omega(P(\rho)-P(\rho_s))(\rho-\rho_s)dxdt+
\mu\int^\infty_0\psi\int_\Omega\nabla
u:\nabla\mathcal{B}[\rho-\rho_s]dxdt\\[3mm]
&\displaystyle+(\lambda+\mu)\int^\infty_0\psi\int_\Omega
\textrm{div}u(\rho-\rho_s)dxdt=0,
\end{array}
\end{equation}
Multiplying (\ref{3.4}) by a negative constant $-\sigma$ with
$0<\sigma\ll1$ and summing up the resulting equation with the energy
inequality (\ref{3.2}), we get
\begin{equation}\label{3.5}
\begin{array}[b]{rl}
&\displaystyle-\int_0^\infty\psi_t\int_\Omega\big(\frac12\rho|u|^2
+\rho\int^\rho_{\rho_s}\frac{h^\gamma-\rho^\gamma_s}{h^2}dh-\sigma
\rho u\mathcal{B}[\rho-\rho_s]\big)dxdt\\[3mm]
&\displaystyle+\int^\infty_0\psi\int_\Omega\big(\mu|\nabla
u|^2+(\lambda+\mu)(\textrm{div}u)^2\big)dxdt-\sigma\int^\infty_0\psi\int_\Omega\rho
u\mathcal{B}[\textrm{div}(\rho u)]dxdt\\[3mm]
&\displaystyle+\sigma\int^\infty_0\psi\int_\Omega\rho u\otimes
u:\nabla\mathcal{B}[\rho-\rho_s]dxdt+\sigma
\int^\infty_0\psi\int_\Omega(\rho^\gamma-\rho_s^\gamma)(\rho-\rho_s)dxdt\\[3mm]
&\displaystyle-\sigma\mu\int^\infty_0\psi\int_\Omega\nabla
u:\nabla\mathcal{B}[\rho-\rho_s]dxdt-\sigma(\lambda+\mu)\int^\infty_0\psi\int_\Omega
\textrm{div}u(\rho-\rho_s)dxdt\leq0.
\end{array}
\end{equation}
Let
$$
V_\sigma=\int_\Omega\big(\frac12\rho|u|^2
+\rho\int^\rho_{\rho_s}\frac{h^\gamma-\rho^\gamma_s}{h^2}dh-\sigma
\rho u\mathcal{B}[\rho-\rho_s]\big)dx,
$$
and
$$
\begin{array}[b]{rl}
W_\sigma=&\displaystyle\int_\Omega\big(\mu|\nabla
u|^2+(\lambda+\mu)(\textrm{div}u)^2\big)dx-\sigma\int_\Omega\rho
u\mathcal{B}[\textrm{div}(\rho u)]dx\\[3mm]
&\displaystyle+\sigma\int_\Omega\rho u\otimes
u:\nabla\mathcal{B}[\rho-\rho_s]dx +\sigma
\int_\Omega(\rho^\gamma-\rho_s^\gamma)(\rho-\rho_s)dx\\[3mm]
&\displaystyle-\sigma\mu\int_\Omega\nabla
u:\nabla\mathcal{B}[\rho-\rho_s]dx-\sigma(\lambda+\mu)\int_\Omega
\textrm{div}u(\rho-\rho_s)dx.
\end{array}
$$
Using the fact $\rho\leq\bar{\rho}$ and the properties of operator
$\mathcal{B}$, we have
\begin{equation}\label{3.6}
\big|\int_\Omega-\sigma \rho u\mathcal{B}[\rho-\rho_s]\big)dx\big|
\leq\frac{\sigma}{2}\int_\Omega\rho|u|^2dx+\frac{\sigma\bar{\rho}c(\Omega)}{2}\int_\Omega(\rho-\rho_s)^2dx.
\end{equation}
Then it follows from $(\ref{3.6})$ and Lemma \ref{lem6} that
\begin{equation}\label{3.7}
c_0(\sigma, \bar{\rho},
\Omega)\int_\Omega\rho|u|^2+(\rho-\rho_s)^2dx\leq V_\sigma\leq
c_1(\sigma, \bar{\rho}, \Omega)\int_\Omega|u|^2+(\rho-\rho_s)^2dx,
\end{equation}
provided $\sigma$ is small enough.

On the other hand, by H\"{o}lder's inequality, we have
$$
\begin{array}[b]{rl}
\big|-\sigma\int_\Omega\rho u\mathcal{B}[\textrm{div}(\rho
u)]dx\big|&\leq\sigma\|\rho
u\|_{L^2(\Omega)}\|\mathcal{B}[\textrm{div}(\rho
u)]\|_{L^2(\Omega)}\\[3mm]
&\leq\sigma c(\Omega)\|\rho u\|_{L^2(\Omega)}^2\leq\sigma
c(\Omega)\bar{\rho}^2\|u\|_{L^2(\Omega)}^2,
\end{array}
$$
$$
\begin{array}[b]{rl}
\big|\sigma\int_\Omega\rho u\otimes
u:\nabla\mathcal{B}[\rho-\rho_s]dx\big|&\displaystyle\leq\sigma\bar{\rho}\big(\int_\Omega|u|^{2p}dx\big)^{\frac{1}{p}}
\big(\int_\Omega|\nabla\mathcal{B}[\rho-\rho_s]|^{q}dx\big)^{\frac{1}{q}}\\[3mm]
&\displaystyle\leq\sigma\bar{\rho}c(\Omega)\int_\Omega|\nabla
u|^2dx\big(\int_\Omega|\rho-\rho_s|^{q}dx\big)^{\frac{1}{q}}\\[3mm]
&\displaystyle\leq\sigma\bar{\rho}^2c(\Omega)\int_\Omega|\nabla
u|^2dx,
\end{array}
$$
where we take
$
p=
  \begin{cases}
   \frac{N}{N-2}& \text{if $N\geq 3$},\\
   2& \text{if $N=2$}
  \end{cases}
$ and $\frac{1}{p}+\frac{1}{q}=1$.

$$
\begin{array}[b]{rl}
\big|-\sigma\mu\int_\Omega\nabla
u:\nabla\mathcal{B}[\rho-\rho_s]dx\big|&\displaystyle\leq\frac{\mu}{4}\int_\Omega|\nabla
u|^2dx+\sigma^2\mu\int_\Omega|\nabla\mathcal{B}[\rho-\rho_s]|^2dx\\[3mm]
&\displaystyle\leq\frac{\mu}{4}\int_\Omega|\nabla u|^2dx+\sigma^2\mu
c(\Omega)\int_\Omega|\rho-\rho_s|^2dx,
\end{array}
$$
$$
\big|-\sigma(\lambda+\mu)\int_\Omega
\textrm{div}u(\rho-\rho_s)dx\big|\leq\frac{\lambda+\mu}{4}\int_{\Omega}|\textrm{div}u|^2dx
+\sigma^2(\lambda+\mu)\int_\Omega|\rho-\rho_s|^2dx,
$$
and it is easy to see that,
$$
\sigma\int_\Omega(\rho^\gamma-\rho^\gamma_s)(\rho-\rho_s)dx\geq
\sigma c\int_\Omega|\rho-\rho_s|^2dx.
$$
In view of the above five estimates, we have
\begin{equation}\label{3.8}
W_\sigma\geq c_2(\sigma, \bar{\rho},
\Omega)\int_\Omega|u|^2+(\rho-\rho_s)^2dx,
\end{equation}
provided $\sigma$ is small enough.

Therefore, from (\ref{3.7}) and (\ref{3.8}), we deduce that for
appropriate selected $\sigma\ll1$, there exists positive constant
$C(\bar{\rho},\Omega)$, such that
\begin{equation}\label{3.9}
C(\bar{\rho},\Omega)V_\sigma\leq W_\sigma.
\end{equation}
Combining (\ref{3.9}) with (\ref{3.5}), we obtain
\begin{equation}\label{3.10}
-\int_0^\infty\psi_t(t)V_\sigma(t)dt+C(\bar{\rho},\Omega)\int_0^\infty\psi(t)V_\sigma(t)dt\leq0,
\end{equation}
for any $\psi\in\mathscr{D}(0,\infty)$ whit $\psi\geq0.$ \par Let
$[\alpha,\beta]$ be any compact subset of $(0,\infty)$, taking
$\psi(t)=\eta_\epsilon(t-\cdot)$ in (\ref{3.10}), we have
\begin{equation}\label{3.11}
\partial_t[V_\sigma]_\epsilon+C(\bar{\rho},\Omega)[V_\sigma]_\epsilon\leq0,\
\textrm{a.e.}\ t\in[\alpha,\beta],
\end{equation}
provided $\epsilon$ is small enough.

Thus
$$
[V_\sigma]_\epsilon(t)\leq[V_\sigma]_\epsilon(s)\exp\{-C(\bar{\rho},\Omega)(t-s)\}
$$
for a.e. $0<s<t<\infty$, according to (\ref{3.10}). Recalling that
$V_\sigma(t)\in L_{loc}^{\infty}(0,\infty)$, letting $\epsilon\to0$,
we have
\begin{equation}\label{3.12}
V_\sigma(t)\leq V_\sigma(s)\exp\{-C(\bar{\rho},\Omega)(t-s)\}\leq
C(E_0, \bar{\rho})\exp\{-C(\bar{\rho},\Omega)t\},
\end{equation}
where $E_0$ denotes the initial energy.

Consequently, (\ref{1.6}) follows from (\ref{3.7}) and (\ref{3.12})
immediately. This completes the proof of Theorem\ref{thm1}.

{\hfill $\square$\medskip}

\section*{Acknowledgements}   This work is partially supported by NSFC grant No.10871175,
10931007, 10901137, Zhejiang Provincial Natural Science Foundation
of China Z6100217,  and SRFDP No. 20090101120005.

\end{document}